\documentclass[12pt]{amsproc}
\usepackage{amsmath,amsthm}
\usepackage{amsfonts,amssymb}

\def\CH{{\mathcal H}}

\def\C{\mathbb{C}}

\def\N{\mathbb{N}}

\def\R{\mathbb{R}}
\def\CF{\mathcal{F}}

 \newtheorem{thm}{ \bf Theorem}[section]
 
 \newtheorem{lem}[thm]{ \bf Lemma}
 \newtheorem{prop}[thm]{ \bf Proposition}

\newtheorem{exam}{Example}[section]

\newcommand{\be}{\begin{equation}}
\newcommand{\ee}{\end{equation}}
\newcommand{\bea}{\begin{eqnarray}}

\newcommand{\eea}{\end{eqnarray}}
\newcommand{\Bea}{\begin{eqnarray*}}
\newcommand{\Eea}{\end{eqnarray*}}

\newcounter{cnt1}
\newcounter{cnt2}
\newcounter{cnt3}
\newcommand{\blr}{\begin{list}{$($\roman{cnt1}$)$}
 {\usecounter{cnt1} \setlength{\topsep}{0pt}
 \setlength{\itemsep}{0pt}}}
\newcommand{\bla}{\begin{list}{$($\alph{cnt2}$)$}
 {\usecounter{cnt2} \setlength{\topsep}{0pt}
 \setlength{\itemsep}{0pt}}}
\newcommand{\bln}{\begin{list}{$($\arabic{cnt3}$)$}
 {\usecounter{cnt3} \setlength{\topsep}{0pt}
 \setlength{\itemsep}{0pt}}}
\newcommand{\el}{\end{list}}

\date{}
\begin{document}

\title[ Hardy class and Hermite expansions ]
{ On the Hermite expansions of functions\\
from Hardy class
\vskip 1.5em { \tt by} }
\author[Garg and Thangavelu]{ Rahul Garg and Sundaram Thangavelu}
\address{Department of Mathematics\\
Indian Institute of Science\\
Bangalore 560 012, India. {\it e-mail~:} {\tt
rahulgarg@math.iisc.ernet.in, veluma@math.iisc.ernet.in}}

\keywords{ Bargmann transform, Hermite functions, Fourier-Wigner transform,
Laguerre functions }
\subjclass{Primary: 42C15 ; Secondary: 42B35, 42C10, 42A56}
\thanks{The second author is supported in part by J. C. Bose Fellowship from
the Department of Science and Technology}

\begin{abstract}
Considering functions $ f $ on $ \R^n $ for which both $ f $ and $ \hat{f} $
are bounded by the Gaussian $ e^{-\frac{1}{2}a|x|^2} , 0 < a < 1 $
we show that their Fourier-Hermite coefficients have exponential decay.
Optimal decay is obtained  for $ O(n)-$finite functions thus extending the
one dimensional result of Vemuri \cite{V}.
\end{abstract}

\maketitle

\section{\bf Introduction }
\setcounter{equation}{0}

Consider the normalised Hermite functions $ \Phi_\alpha, \alpha
\in \N^n $ on $ \R^n $ which  are eigenfunctions of the Hermite
operator $ H = -\Delta+|x|^2 $ with eigenvalues $ (2|\alpha|+n).$
They form an orthonormal basis for $ L^2(\R^n) $ so that every $ f
\in L^2(\R^n) $ has the expansion
$$  f = \sum_{\alpha \in \N^n} (f,\Phi_\alpha) \Phi_\alpha~ .$$
When the Hermite coefficients of $ f $ has exponential decay, say
$|(f,\Phi_\alpha)| \leq C e^{-(2|\alpha|+n)t}$, for some $t > 0 $,
then by Mehler's formula (see \cite{Th1}) it can be easily shown
that $ f $ satisfies the estimate
$$ |f(x)| \leq C e^{-\frac{1}{2}\tanh(t)|x|^2} .$$ As  $ \Phi_\alpha $ are
also eigenfunctions of the Fourier transform with eigenvalues $
(-i)^{|\alpha|} $, it follows that $|(\hat{f},\Phi_\alpha)| \leq C
e^{-(2|\alpha|+n)t} $ and hence $ \hat{f} $ also satisfies the
same estimate as $ f.$

However, it is possible to prove better estimates for $ f $ and $
\hat{f}.$ The assumption on $ (f,\Phi_\alpha) $ together with the
asymptotic properties of holomorphically extended Hermite
functions  lead us to the fact that $ f $ extends to $ \C^n $ as
an entire function and satisfies
$$ |f(x+iy)| \leq C_m(1+|x|^2+|y|^2)^{-m}
e^{-\frac{1}{2}\tanh(2s)|x|^2+ \frac{1}{2}\coth(2s)|y|^2} $$ for every $ m \in \N $ and $ 0 < s <t.$
And a similar estimate holds for $ \hat{f} $ as well. Indeed, under
the assumption on $ (f,\Phi_\alpha) $ the entire function $ f(z) $ belongs
to the Hermite
Bergman space $ \CH_s(\C^n) $ consisting of entire functions which are square
integrable with respect to the weight function
$$ U_s(x,y) = e^{\tanh(2s)|x|^2-\coth(2s)|y|^2} $$
for every $ s < t $ and hence as shown in \cite{RT} the functions
$ f(z) $ and $\hat{f}(z) $ both satisfy the above estimate.

Suppose we only know that $ f $  and $ \hat{f} $ are bounded on $ \R^n $ by
the Gaussian $ e^{-\frac{1}{2}\tanh(2t)|x|^2}.$ We would like to know if
these conditions in turn imply some exponential decay of the Hermite
coefficients of $ f .$ It will be so if we can prove that $ f(z) $ satisfies
$$ |f(x+iy)|^2 \leq C e^{-\tanh(2s)|x|^2+\coth(2s)|y|^2} $$
for some $ s > 0.$ Under the assumption on $ f $ and $ \hat{f} $ it is clear,
from the Fourier inversion formula, that $ f $ extends to $ \C^n $ as an
entire function which satisfies
$$ |f(x+iy)| \leq C e^{\frac{1}{2}\coth(2s)|y|^2} .$$ But a priori it is not
at all clear if $ f(x+iy) $ has any decay in $ x.$ In this article we
address the problem of estimating $ f $ on $ \C^n .$

This problem has connections with a classical theorem of Hardy
\cite{H} proved in 1933 which says that a function  $ f $ and its
Fourier transform $ \hat{f} $ both cannot have arbitrary Gaussian
decay. The precise statement is as follows. For a function $ f \in
L^1(\R^n) $, let
$$ \hat{f}(\xi) = (2\pi)^{-n/2} \int_{\R^n} f(x) e^{-i x\cdot \xi} dx $$
be its Fourier transform. Suppose
$$ |f(x)| \leq C e^{-a|x|^2},~~~~ |\hat{f}(\xi)| \leq C e^{-b |\xi|^2} $$
for some positive constants $ a $ and $ b.$ Then $ f = 0 $ when $
ab >1/4 $ and $ f(x) = C e^{-a|x|^2} $ when $ ab = 1/4.$ Moreover,
there are infinitely many linearly independent functions
satisfying both conditions when $ ab < 1/4.$ Examples of such
functions are provided by  the Hermite functions $ \Phi_\alpha.$

Hardy's theorem has received considerable attention over the last fifteen
years or so as can be seen from the large number of papers written on the
theorem, see e.g. the monograph \cite{Th3} and the references therein. However,
all the works so far have treated only the case $ ab \geq 1/4 $ in various
set-ups. The case $ ab < 1/4 $ did not receive any closer study until
recently where in \cite{V} Vemuri has looked
at functions satisfying Hardy conditions with $ a = b < 1/2 .$ By a very
clever use of Bargmann transform he has proved the following characterisation
of such functions.

\begin{thm} Suppose $ f \in L^1(\R) $ satisfies the conditions
$$ |f(x)| \leq C e^{-\frac{1}{2}a x^2},~~~ |\hat{f}(\xi)| \leq C
e^{-\frac{1}{2} a \xi^2} $$ for some $ 0<a<1 $. Then the
Fourier-Hermite coefficients of $ f $ satisfy $ |(f,\Phi_k)| \leq
C (2k+1)^{-\frac{1}{4}}e^{-(2k+1)t/2}$, where t is determined by
the condition $ a = \tanh(2t).$
\end{thm}

In \cite{V} the author has considered functions of one variable
only. A natural question is whether a similar result is true for
functions on $ \R^n.$  The proof in \cite{V}, like many other
proofs of Hardy-type theorems, depends on Phragmen-Lindelof
maximum principle which is essentially a theorem in one complex
variables. If we consider functions $ f $ which are tensor
products of one dimensional functions, then an analogue  of
Theorem 1.1 follows easily. More generally, the arguments in
\cite{V} can be used to prove the following result. We state the
result in terms of the Hermite projection operators $ P_k $ which
are defined by
$$ P_kf = \sum_{|\alpha|=k} (f,\Phi_\alpha) \Phi_\alpha $$ for any $f \in
L^2(\R^n).$ We refer to \cite{Th1} for more about Hermite
expansions.

\begin{thm} Suppose $ |f(x)| \leq C e^{-\frac{1}{2}a|x|^2} $ and for any
$ j = 1,2,...,n$, $ |\CF_jf(\xi)| \leq C e^{-\frac{1}{2}a|x|^2} $ where
$ {\CF_jf} $ is the partial Fourier transform of $ f $ in any set of $ j $
variables. Then we have the estimates
$ \|P_kf\|_2 \leq C (2k+n)^{\frac{n-2}{4}}e^{-(2k+n)t/2} $ where
$ a = \tanh(2t).$
\end{thm}

There are strong reasons to believe that the result is true for all functions
satisfying the Hardy conditions. However, at present we do not know how to prove
the result. Nevertheless, we have the following slightly weaker result.

\begin{thm}Suppose $ f \in L^1(\R^n) $ satisfies the estimates
$$  |f(x)| \leq C e^{-\frac{1}{2}a|x|^2},~~ |\hat{f}(x)| \leq C
e^{-\frac{1}{2}a|x|^2} $$ for some $ 0<a<1. $ Then $ \|P_kf\|_2
\leq C(2k+n)^{\frac{n-1}{4}} e^{-(2k+n)s/2} $ where $ s $ is
determined by the condition $ \tanh(2s) = a/2.$
\end{thm}

We prove this theorem in Section 4 by relating the Hermite
projections $ P_kf $ with the Fourier-Wigner transform $ V(f,f) $
and appealing to a version of Hardy's theorem for Hankel
transform. Since the Fourier transform of a radial function
reduces to a Hankel transform, Theorem 1.1 can be shown to be true
for all radial functions. More generally, we can prove the same
for all $ O(n)-$finite functions in $ L^2(\R^n).$ In other words,
Theorem 1.1 remains true for all functions whose restrictions to
the unit sphere $ S^{n-1} $ have only finitely many terms in their
spherical harmonic expansions.

\begin{thm}Suppose $ f \in L^1(\R^n) $ satisfies the same conditions
as in the previous theorem. If we further assume that $ f $ is $ O(n)-$finite,
then $ \|P_kf\|_2  \leq C(2k+n)^{\frac{n-2}{4}}
e^{-(2k+n)t/2}$ where $ a = \tanh(2t).$
\end{thm}

We prove this theorem in Section 5 by studying a vector valued
Bargmann transform. Let us define the Hardy class $ H(a), 0 < a <
1 $ as the set of all functions $ f $ satisfying the Hardy
conditions  in Theorem 1.3. We are interested in estimating the
Hermite coefficients of $ f $ from $ H(a).$ This problem has been
completely solved in the one dimensional case by Vemuri \cite{V}.
A work closely related to this article is the paper by Janssen and
Eijndhoven \cite{J} where they have studied growth of Hermite
coefficients in one dimension. Here we treat the higher
dimensional case. It would also be interesting to find the precise
relation between Hardy conditions and the membership in
Hermite-Bergman spaces $ \CH_t(\C^n)$.

\section{\bf Preliminaries }

In this section we set up the notations and collect relevant
results about Hermite functions, Fourier-Wigner and Hankel
transforms. We closely follow the notations used in \cite{Th2} and
\cite{F} and we refer to the same for the proofs and any
unexplained terminology. Writing down the Hermite expansion of $ f
\in L^2(\R^n) $ as $ f = \sum_{k=0}^\infty P_kf $, the Plancherel
theorem  reads as $ \|f\|_2^2 = \sum_{k=0}^\infty \|P_kf\|_2^2 .$

The Fourier-Wigner transform of two functions $ f, g \in L^2(\R^n) $ is a function
on $ \C^n $ defined by
$$ V(f,g)(x+iy) = (2\pi)^{-n/2}\int_{\R^n} e^{i (x\cdot \xi+\frac{1}{2}x\cdot y)}
f(\xi+y)\overline{g(\xi)}  d\xi.$$ We make use of the identity
(see \cite{Th1})
$$ \int_{\C^n} V(f_1,g_1)(z) \overline{V(f_2,g_2)(z)} dz = (f_1,f_2)(g_2,g_1).$$
for any $ f_i,g_i \in L^2(\R^n) $.

The special Hermite functions $ \Phi_{\alpha \beta} =
V(\Phi_\alpha,\Phi_\beta) $ form an orthonormal basis for $
L^2(\C^n).$ We observe that
$$ \sum_{|\alpha| = k}|(f,\Phi_\alpha)|^2 = \int_{\C^n} V(f,f)(z)
\sum_{|\alpha| = k} \overline{\Phi_{\alpha \alpha}(z)}~ dz.$$ If
we let $ \varphi^{n-1}_k(z) $ stands for the Laguerre function $
L_k^{(n-1)} (\frac{1}{2}|z|^2)e^{-\frac{1}{4}|z|^2} $ then we know
that
$$ \sum_{|\alpha| = k} \Phi_{\alpha,\alpha}(z) = (2\pi)^{-n/2}\varphi^{n-1}_k(z) $$
and therefore we get the useful relation
$$ \|P_kf\|_2^2 = (2\pi)^{-n/2} \int_{\C^n} V(f,f)(z) \varphi^{n-1}_k(z)~dz $$
which will be used in Section 4. The same idea has been used in \cite{J} in
the study of growth of Hermite coefficients.

In Section 5 we will make use of a Hecke-Bochner type formula for
the Hermite projection operators. Let $ P $ be a harmonic
polynomial which is homogeneous of degree $ m $, called a solid
harmonic. It is well known that if $ f $ is radial, then the
Fourier transform of $ fP $ is again of the same form, viz.
$\widehat{fP} = F P $ where $ F $ is given by a Hankel transform.
A similar result is true for the Hermite projections. Let $
L_k^\delta $ stand for Laguerre polynomials of type $ \delta $
which are defined by the generating function identity
$$ \sum_{k=0}^\infty L_k^\delta(x)e^{-\frac{1}{2}x}r^k = (1-r)^{-\delta-1}
e^{-\frac{1}{2}\frac{1+r}{1-r}x} $$ for $ |r| <1, x > 0.$ Define
$$ R_k^\delta(f) = 2\frac{\Gamma(k+1)}{\Gamma(k+\delta+1)}
\int_0^\infty f(s) \psi_k^\delta(s) s^{2\delta+1} ds $$ where the Laguerre
functions $ \psi_k^\delta $ are defined by
$$ \psi_k^\delta(s) = L_k^\delta(s^2)e^{-\frac{1}{2}s^2}.$$ With these notations we have

\begin{prop} Suppose $ f \in L^2(\R^n) $ be such that $ f = g P $ where $ g $
is radial and $ P $ is a solid harmonic of degree $ m.$ Then $ P_jf = 0 $
unless $ j = 2k+m $ in which case
$$ P_{2k+m}f(x) = R_k^{n/2+m-1}(g)P(x)\psi_k^{\frac{n}{2}+m-1}(|x|).$$
\end{prop}

The restrictions of solid harmonics to $ S^{n-1} $ are called
spherical harmonics. Let $ \{ Y_{mj}: 1 \leq j \leq d_m, m \in \N
\} $ be an orthonormal basis for $ L^2(S^{n-1}) $ consisting of
spherical harmonics. Given $ f \in L^2(\R^n) $ we have the
expansion
$$ f(r\omega) = \sum_{m=0}^\infty \sum_{j=1}^{d_m} f_{mj}(r)Y_{mj}
(\omega) $$ where $ f_{mj}$ are the spherical harmonic
coefficients of $ f $ defined by
$$ f_{mj}(r) = \int_{S^{n-1}}f(r\eta)Y_{mj}(\eta)~ d\eta.$$
The above proposition leads to the formula
$$ P_{2k}f(x) = \sum_{m=0}^k \sum_{j=1}^{d_{2m}}
R_{k-m}^{\frac{n}{2}+2m-1}(\tilde{f}_{2m,j})~
\psi_{k-m}^{\frac{n}{2}+2m-1}(r)~r^{2m}Y_{2m,j}(\omega).$$ where $
\tilde{f}_{m,j}(r)=r^{-m}f_{mj}(r) $. A similar formula can be
written for $ P_{2k+1}f $ as well. The functions $ \psi_k^\delta $
form an orthogonal system in $ L^2(\R_+, r^{2\delta+1}dr) $ and
suitably normalised they form an orthonormal basis.

\section{\bf Bargmann transform and Hardy's theorem}
\setcounter{equation}{0}

For the convenience of the readers we briefly recall the argument
used by Vemuri \cite{V} in proving Theorem 1.1. As we have already
mentioned we will be using variants of the same arguments, so it
will help fixing the ideas. Recall that the Bargmann transform $ B
$ defined by
$$ Bf(z)= e^{- \frac {1}{4}z^2} \int_{\mathbb{R}^n}f(x)~
e^{- \frac{1}{2}|x|^2} e^{z \cdot x } dx $$ for $z \in \C^n $ is
an isometric isomorphism from $ L^2(\R^n) $ onto the Fock space
consisting of entire functions on $ \C^n $ that are square
integrable with respect to the Gaussian $ e^{-\frac{1}{2}|z|^2}$,
see Bargmann \cite{B}. It takes the Hermite functions $
\Phi_\alpha $ onto the monomials $ \zeta_\alpha(z) = \big(
\frac{2^\alpha \alpha!}{\pi^{n/2}} \big)^{-\frac{1}{2}} z^\alpha.$
Moreover, it has the interesting property that $
B\widehat{f}(z)~=~ (2\pi)^{\frac{n}{2}}Bf(-iz).$

If $ f $ satisfies the Gaussian estimate $ f(x) = O(e^{-\frac{1}{2}a|x|^2}) $
then from the definition of $ B $ it follows that
$$ |Bf(w)| \leq  C(1+a)^{-\frac {n}{2}}~\exp\big(\frac{v^2+\mu u^2}{4} \big) $$
where $ \mu = \frac{1-a}{1+a}~,~w = u+iv $ and $ u^2 =
\sum_{j=1}^n u_j^2 $ etc. The relation $
B\widehat{f}(z)~=~(2\pi)^{\frac{n}{2}}Bf(-iz) $ then leads to
$$ |Bf(w)| \leq  C(1+a)^{-\frac {n}{2}}~\exp\big(\frac{u^2+\mu v^2}{4} \big) .$$
When $ n =1 $ and taking $ w = re^{i\theta} $ we get
$$ |Bf(w)| \leq  C(1+a)^{-\frac {1}{2}}
\exp\Big(\frac{(\mu +(1-\mu)\sin^2\theta)r^2}{4} \Big) $$ and
$$ |Bf(w)| \leq  C(1+a)^{-\frac {1}{2}}
\exp\Big(\frac{(\mu +(1-\mu)\cos^2\theta)r^2}{4} \Big) .$$ A
Phragmen-Lindelof argument then leads to the estimate
$$ |Bf(w)| \leq  C(1+a)^{-\frac {1}{2}}\exp\big(\frac{\sqrt{\mu}}{4}r^2 \big).$$

If $ c_k $ are the Taylor coefficients of $ Bf $ then Cauchy's estimates lead to
$$ |c_k| \leq C(1+a)^{-\frac {1}{2}}\exp\big(\frac{\sqrt{\mu}}{4}r^2 \big)r^{-k} $$
and optimizing with respect to $ r $ we can get
$$ |c_k| \leq  C(1+a)^{-\frac {1}{2}}\big(  \frac{e\sqrt{\mu}}{2k} \big)^{k/2}.$$
Since $ c_k $ are related to the Hermite coefficients of $ f $, we
get a slightly weaker form of Theorem 1.1. For the argument
leading to Theorem 1.1 we refer to \cite{V}. In the $
n-$dimensional case it is possible to use the same arguments to
prove Theorem 1.2 under the extra assumptions made in the
hypothesis. We leave the details to the reader.

\section{\bf Hardy's theorem for Hankel transform\\
and a proof of Theorem 1.3}
\setcounter{equation}{0}

For any  $\delta > -\frac{1}{2} $ we define the Hankel transform $
H_\delta $ on $  L^1(\mathbb{R}^+,r^{2\delta + 1}dr) $ by
$$ H_\delta f(r) = \int_0^\infty f(s) \frac{J_\delta(rs)}{(rs)^\delta}
s^{2\delta+1} ds.$$ It is well known that $ H_\delta $ extends to
$  L^2(\mathbb{R}^+,r^{2\delta + 1}dr) $ as a unitary operator and
the inversion formula is given by
$$ f(r) = \int_0^\infty H_\delta f(s) \frac{J_\delta(rs)}{(rs)^\delta}
s^{2\delta+1} ds $$ for all $ f $ for which $ H_\delta f $ is
integrable with respect to $ s^{2\delta+1}ds.$ Moreover, it is
known that $ H_{\delta}\psi_{k}^{\delta} =(-1)^{k}~
\psi_{k}^{\delta} .$ We will make use of this fact in what
follows.

An analogue of Hardy's theorem (i.e. the case $ ab \geq 1/4 $) is known for
the Hankel transform as well. We now prove an analogue of Theorem 1.1 for
the Hankel transform.

\begin{thm} Let $ f \in  L^1(\mathbb{R}^+,r^{2\delta + 1}dr) $ be such
that both $ f $ and $ H_\delta f $ satisfy the Hardy condition
with $ a = \tanh(2t).$ Then the Laguerre coefficients of $ f $
satisfy the following estimate:
$$ |(f,\psi_k^\delta)| \leq C (1+a)^{-\delta}
(4k+2\delta+1)^{\delta}e^{-2tk}.$$
\end{thm}

The proof of this theorem is similar to that of Theorem 1.1 given
in \cite{V}. We just need to replace the Bargmann transform by
another transform adapted to the Hankel transform. We now proceed
to define this transform which we denote by $ U_\delta.$  For $ f
\in L^{^2}(\mathbb{R}^+,r^{2\delta+1}~dr) $ we let
$$ U_\delta f(w) = e^{\frac{w^2}{4}} \int_0^\infty f(s)
~\frac{J_\delta(iws)}{(iws)^\delta}e^{-\frac{1}{2}(w^2+s^2)}
~s^{2\delta+1}ds.$$ It is clear that $ U_\delta f $ extends to $
\C $ as an even entire function of $ w.$ Moreover, the generating
function identity
$$ \sum_{k=0}^{\infty}\frac{L_k^\delta(x)}{\Gamma(k+\delta+1)}~w^k
= e^w(xw)^{-\frac{\delta}{2}}J_\delta\big(2(xw)^{\frac{1}{2}}\big)
$$ satisfied by the Laguerre polynomials can be rewritten as
$$
\sum_{k=0}^{\infty}\frac{(-1)^k~2^{-2k}}{\Gamma(k+\delta+1)}
\psi_k^\delta(r)~w^{2k} ~=~ 2^\delta
e^{\frac{w^2}{4}}~\frac{J_\delta(irw)}{(irw)^{\delta}}
~{e^{-\frac{1}{2}(r^2+w^2)}}.$$ In view of this we have
$$ U_\delta f(w) = 2^{-\delta}\sum_{k=0}^{\infty}\frac{(-1)^k~
2^{-2k}}{\Gamma(k+\delta+1)} (f,\psi_k^\delta)~w^{2k}.$$ This
shows that the transformation $ U_\delta $ takes the Laguerre
functions $ \psi_k^\delta $ onto constant multiples of the
monomials $ w^{2k}.$ We also have the relation $ U_\delta H_\delta
f(w) = U_\delta f(-iw) $ which follows from the fact that $
H_\delta \psi_k^\delta = (-1)^k \psi_k^\delta.$

The image of $  L^2(\mathbb{R}^+,r^{2\delta+1}~dr) $ under the
transform $ U_\delta $ is known to be a weighted Bargmann space,
see \cite{C}. Indeed, if we let
$$ h(w) = \frac{2^\delta}{\pi} \big(\frac{|w|^2}{2}\big)
^{2\delta + 1} K_{\delta +\frac{1}{2}}(\frac{|w|^2}{2}) $$ where
$$ K_\delta(z) = (\frac{\pi}{2z})^{\frac{1}{2}} \frac{e^{-z}}
{{\Gamma(\delta + \frac{1}{2})}}\int_0^\infty e^{-t} t^{\delta -
\frac{1}{2}} \big(1 + \frac{t}{2z} \big)^ {\delta - \frac{1}{2}}
dt $$ then the image is precisely the Hilbert space of even entire
functions that are square integrable with respect to $ h(w) dw $
(see Cholewinski \cite{C}). As $ h(w) $ is radial it is clear that
$ w^{2k} $ form an orthogonal system with respect to $ h(w) dw.$
Moreover, it can be shown that (see \cite{C})
$$ \int_{\C} |w|^{4k} h(w) dw = 2^{1+2\delta}~2^{4k}
\Gamma(k+1)~\Gamma(k+\delta+1).$$
Thus, if we let
$$ \zeta_k(w)= (-1)^{k}\Big(2^{1+2\delta+4k}~
\Gamma(k+1)~\Gamma(k+\delta+1)\Big)^{-\frac{1}{2}}w^{2k} $$ then $
\zeta_k $ form an orthonormal basis for the  image of $
L^2(\mathbb{R}^+,r^{2\delta+1}~dr) $ under $ U_\delta $ and $
U_\delta \Big(\big(2\frac{\Gamma(k+1)}{\Gamma(k+\delta+1)}\big)^
{\frac{1}{2}}\psi_{_k}^{^\delta}\Big)(w) = \zeta_k(w).$

We can now proceed as in Vemuri \cite{V} with $ U_\delta $ playing the role of the
Bargmann transform to prove Theorem 4.1.

We now use Hardy's theorem for the Hankel transform to prove
Theorem 1.3. Let $ \mathfrak{F}_s $ stands for the symplectic
Fourier transform. We need the following estimate on the
Fourier-Wigner transform $ V(f,f) $ when $ f $ and $ \hat{f} $
satisfy Hardy conditions.

\begin{prop} Let  $f \in L^2(\R^n) $ satisfies $|f(x)| \leq Ce^{-a|x|^2}\
$ and $|\widehat{f}(\xi)| \leq Ce^{-a|\xi|^2}$ for some $ C > 0 $
and $a > 0.$ Then $|V(f,f)(z)| \leq C_na^{-\frac{n}{2}} e^{-\frac
{1}{4}a|z|^2}$ and $|\mathfrak{F}_sV(f,f)(z)| \leq
C_na^{-\frac{n}{2}} e^{-\frac {1}{4}a|z|^2}$, where $ C_n>0 $
depends only on C and n.
\end{prop}
\begin{proof} By definition, for $ z= x+iy \in \mathbb{C}^n$
$$ V(f,f)(z)=(2\pi)^{-\frac{n}{2}}\int_{{\mathbb{R}}^n}
e^{i(x.\xi+\frac{1}{2}x.y)}f(\xi+y)~\overline{f(\xi)}~d\xi.$$ An
easy calculation using Fourier inversion shows that
$$ V(\widehat{f},\widehat{f})(z) = (2\pi)^{n}~V(f,f)(iz).$$
From the definition it follows that $|V(f,f)(z) | $ is bounded by
$$ (2\pi)^{-\frac{n}{2}}\int_{{\mathbb{R}}^n}
|f(\xi+y) | ~| \overline{f(\xi)} | ~d\xi  \leq C
\int_{{\mathbb{R}}^n}  e^{-a |\xi+y|^2} e^{-a|\xi|^2}~d\xi.$$ The
last integral is equal to
$$ C e^{-\frac{1}{2}a|y|^2}
\int_{{\mathbb{R}}^n} e^{-2a | \xi+ \frac{y}{2}|^2}~d\xi
= C_n a^{-\frac{n}{2}} e^{-\frac{1}{2}a|y|^2}.$$
Replacing $f$ by $\widehat{f}$, we also get
$$ |V(\widehat{f},\widehat{f})(z) | \leq C_n a^{-\frac{n}{2}}
e^{-\frac{1}{2}a|y|^2}.$$
And thus the relation $V(f,f)(z)=(2\pi)^{-n} V(\widehat{f},\widehat{f})(-iz)$
gives
$$ |V(f,f)(z) |  \leq  C_n a^{-\frac{n}{2}} e^{-\frac{1}{2}a|x|^2}.$$
Combining these two, we get
$$ |V(f,f)(z) | \leq C_n a^{-\frac{n}{2}} e^{-\frac{1}{4}a|z|^2}.$$
The above calculation together with the relation $\mathfrak{F}_sV(f,f)(z)=
(8\pi)^{\frac{n}{2}}V(f,\widetilde{f})(z)$ implies that
$$ |\mathfrak{F}_sV(f,f)(z) | \leq C_n
a^{-\frac{n}{2}}e^{-\frac{1}{4} a|z|^2}.$$
This completes the proof of the proposition.
\end{proof}

In view of the expression for the norm of $ P_k $ in terms of the
Laguerre coefficients of $ V(f,f) ,$ in order to prove Theorem 1.3
we only need to prove the following result.

\begin{thm} Suppose $ f \in L^2(\R^n) $ is such that $ |V(f,f)(z)| \leq C
e^{-\frac{1}{8}a|z|^2} $ and
$|\mathfrak{F}_sV(f,f)(z)| \leq C e^{-\frac {1}{8}a|z|^2}.$
Then
$$ |\int_{\C^n} V(f,f)(z) \varphi_k^{n-1}(z) dz | \leq C (2k+n)^{n-1} e^{-(2k+n)s} $$
where $ s $ is determined by $ a/2 = \tanh(2s).$
\end{thm}
\begin{proof} As $ \varphi_k^{n-1}(z) $ is radial, recalling the definition of
$ \psi_k^{n-1} $ the integral we want to estimate reduces to $
2^n\int_0^\infty F(\sqrt{2}r) \psi_k^{n-1}(r) r^{2n-1} dr $ where
$$ F(r) = \int_{S^{2n-1}} V(f,f)(r\omega) d\omega $$  which clearly satisfies the
estimate $ |F(r)| \leq C e^{-\frac{1}{8}ar^2}.$ If we can show
that the function $ G(r) = F(\sqrt{2}r) $ satisfies the estimate $
|H_{n-1}G(r)| \leq C e^{-\frac{1}{4}a^2} $, then we can appeal to
Theorem 4.1 to get the required estimate. Since $ \varphi_k^{n-1}
$ is an eigenfunction of the symplectic Fourier transform
$$ \int_{\C^n} V(f,f)(z) \varphi_k^{n-1}(z) dz  = (-1)^k \int_{\C^n}
\mathfrak{F}_sV(f,f)(z) \varphi_k^{n-1}(z) dz .$$

We now perform the following calculations: \Bea \int_{S^{2n-1}}
\mathfrak{F}_s {V(f,f)}( \sqrt{2}~r\omega)~d\omega &=&
\int_{S^{2n-1}} \widehat{V(f,f)}(-{\frac{i}{2}} \sqrt{2}~r\omega)~
d\omega\\
&=& \int_{S^{2n-1}} \widehat{V(f,f)}({\frac{r}{\sqrt{2}}} \omega)~d\omega \\
&=& \int_0^\infty F(s) \frac{J_{n-1}({\frac{r}{\sqrt{2}}}s)}
{({\frac{r}{\sqrt{2}}}s)^{n-1}}s^{2n-1}ds \\
&=& 2^n \int_0^{\infty}F(\sqrt{2}s) \frac{J_{n-1}(rs)}{(rs)^{n-1}}
s^{2n-1}ds \\
&=& 2^n \int_0^{\infty}G(s) \frac{J_{n-1}(rs)}{(rs)^{n-1}}
s^{2n-1}ds \\
&=& 2^n (H_{n-1}G)(r)
\Eea
which proves our claim on $ H_{n-1}G(r).$
\end{proof}

If we can improve the estimates in Proposition 4.2 to $|V(f,f)(z)|
\leq C e^{-\frac {1}{2}a|z|^2}$ and $|\mathfrak{F}_sV(f,f)(z)|
\leq C e^{-\frac {1}{2}a|z|^2}$ then we could prove Theorem 4.3
with $ \tanh(2s) = a .$ In fact, one needs only
$$ | \int_{S^{2n-1}} V(f,f)(r\omega) d\omega | \leq C e^{-\frac{1}{2}ar^2} $$
and a similar estimate for $ \mathfrak{F}_sV(f,f) $ which is good
enough to improve Theorem 1.3. But there are some limitations on
the decay of Fourier-Wigner transform due to the uncertainty
principle proved in \cite{G}. The following example shows that
improving the estimates in Proposition 4.2 is not always possible
which means that the proof via Fourier-Wigner transform is not
robust enough to lead to Theorem 1.3.

\begin{exam} Let $a=\frac{1}{\sqrt{2}}$ and consider the function $f \in
L^2(\R^2)$, defined by \Bea
f(x_1,x_2)=e^{-\frac{a}{2}(x_1^2+x_2^2+2ix_1x_2)}. \Eea An easy
calculation (using $ a = \frac{1}{2a}$) shows that \Bea
\widehat{f}(\xi,\eta) =  \frac{\pi \sqrt{2}}{a}
e^{-\frac{1}{4a}(\xi^2 +\eta^2 -2i \xi \eta)} =  2\pi ~
e^{-\frac{a}{2}(\xi^2+\eta^2-2i\xi\eta)} \Eea and with $ z
=(x_1+iy_1,x_2+iy_2) $ \Bea V(f,f)(z) =
\frac{1}{2a}~e^{-\frac{a}{2}|z|^2} ~e^{\frac{1}{2}(x_1 y_2+~x_2
y_1)}. \Eea From the above expression for $ V(f,f) $, it is clear
that the estimate
$$ | V(f,f)(z)| \leq C e^{-\frac{a}{4}|z|^2} $$ is not valid.

For any $ 0<b<\frac{1}{2} $, let us write $$ R_b=\big
\{\omega=(\omega_1,\omega_2,\omega_3,\omega_4) \in S^3 ~ \big| ~
\omega_1\omega_4>b,~ \omega_2\omega_3>0 \big \} .$$ Clearly, $R_b$
is a subset of $S^3$ of positive measure. Let us denote its
measure by $ |R_b|. $ Then, \Bea \int_{S^3} V(f,f)(r\omega)
d\omega &=& \frac{1}{2a}~e^{-\frac{a}{2}r^2} \int_{S^3}
e^{\frac{1}{2}r^2(\omega_{_1}\omega_{_4}+\omega_{_2}\omega_{_3})} d\omega \\
& \geq & \frac{1}{2a}~ e^{-\frac{a}{2}r^2}\int_{R_b}
e^{\frac{1}{2}r^2 (\omega_{_1}\omega_{_4}+\omega_{_2}\omega_{_3})} d\omega \\
& \geq & \frac{1}{2a}|R_b|~e^{-\frac{a}{2}r^2}e^{\frac{b}{2}r^2} \\
&=& \frac{1}{2a}|R_b|~e^{-\frac{a}{4}r^2}
e^{\frac{1}{2}(b-\frac{a}{2})r^2}\Eea Since
$a=\frac{1}{\sqrt{2}}$, one can choose $0<b<\frac{1}{2}$ such that
$b-\frac{a}{2}>0$. Therefore, $ \int_{S^3} V(f,f)(r\omega) d\omega
$ can not be bounded by $ e^{-\frac{a}{4}r^2} $.

However we can show that $ \|P_kf\|_2 $ has the required decay.
Indeed, for any $ 0 < r < 1,$  we have
$$ \sum_{k=0}^\infty r^k \|P_kf\|_2^2 = (2\pi)^{-1} \int_{\C^2} V(f,f)(z)
\left(\sum_{k=0}^\infty r^k \varphi_k^1(z)\right) dz $$ and hence using the
generating function identity
$$ \sum_{k=0}^\infty r^k \varphi_k^1(z) = (1-r)^{-1}e^{-\frac{1}{4}
\frac{1+r}{1-r}|z|^2}  $$ and the explicit expression for $ V(f,f) $
we can calculate that
$$ \sum_{k=0}^\infty r^k \|P_kf\|_2^2  =  \frac{8\pi a(1-r)^{-2}}
{(2a+\frac{1+r}{1-r}) \Big((2a+\frac{1+r}{1-r})-
\frac{1}{(2a+\frac{1+r}{1-r})}\Big)}.$$ Writing $\mu =
\frac{1-a}{1+a}$ and simplifying, the above takes the form
$$ \sum_{k=0}^\infty r^k \|P_kf\|_2^2  = \frac{2\pi}{1+a} \sum_{k=0}^\infty
\mu^k r^{2k}.$$
Comparing coefficients of $ r^k $ we see that $ P_{2k+1}f = 0 $ and
$ \|P_{2k}f\|_2^2 = \frac{2\pi}{1+a}\mu^k $ which is the expected decay.
\end{exam}

\section{\bf A vector valued Bargmann transform\\
and a proof of Theorem 1.4}
\setcounter{equation}{0}

In order to prove Theorem 1.4 we need to study a vector valued
Bargmann transform. For functions $ f $ from $ L^2(\R^n) $
consider
$$ Bf(z, \omega)=e^{- \frac{1}{4}z^2} \int_{\mathbb{R}^n}f(x)~
e^{- \frac{1}{2} |x|^2} e^{zx.\omega} dx $$ where  $z \in
\mathbb{C}$ and $\omega \in S^{n-1}.$ We think of $ Bf $ as an
entire function of the one complex variable taking values in the
vector space $ L^2(S^{n-1}).$ As before, one  can easily verify
that $
B\widehat{f}(z,\omega)~=~(2\pi)^{\frac{n}{2}}Bf(-iz,\omega).$ We
consider functions satisfying the conditions:
\begin{equation}
\left(\int_{S^{n-1}} |f(s\eta)|^2~d \eta \right)^{\frac {1}{2}} \leq C e^{-
\frac{a}{2}s^2},
\end{equation}
\begin{equation}
\left(\int_{S^{n-1}} |\widehat{f}(s\eta)|^2~d\eta \right)^{\frac {1}{2}}\leq C
e^{- \frac{a}{2}s^2}
\end{equation}
for some $a > 0$. The basic estimates on the Bargmann transforms of such functions
are given below.

\begin{prop} If $f \in L^2( {\mathbb{R}}^n)$ be such that
(5.1) and (5.2) are valid for some $ a > 0 $. Then for every $z=u+iv
\in \mathbb{C}$,
$$ \int_{S^{n-1}} |Bf(z,\omega)|^2~d \omega  \leq
C(1+a)^{-n}~\exp\big(\frac{v^2+\mu u^2}{2} \big), $$
$$ \int_{S^{n-1}} |Bf(z,\omega)|^2~d \omega  \leq C
(1+a)^{-n} \exp\big(\frac{u^2+\mu v^2}{2} \big) $$
where $\mu = \frac {1-a}{1+a}$ as before.
\end{prop}
\begin{proof} For $z \in \mathbb{C}$ and $\omega \in S^{n-1}$,
\Bea Bf(z,\omega) &=& e^{- \frac {1}{4}z^2}
\int_{\mathbb{R}^n}f(x)~
e^{- \frac{1}{2}|x|^2} e^{zx.\omega}~dx\\
&=& e^{- \frac {1}{4}z^2} \int_0^{\infty}
\Big(\int_{S^{n-1}}f(s\eta)
~e^{sz\eta .\omega}~d\eta \Big)~e^{-\frac{1}{2} s^2}s^{n-1} ds\\
\Eea
Thus we get the estimate
$$|Bf(z,\omega)| \leq |e^{- \frac {1}{4}z^2}| \int_0^\infty
|T_{sz}f(\omega)|~e^{- \frac{1}{2}s^2}s^{n-1} ds $$ where $$
T_{sz}f(\omega) =\int_{S^{n-1}}f(s\eta)~e^{sz\eta .\omega}~d\eta
$$ If we write $z=u+iv$, then $$ |T_{sz}f(\omega)| \leq
\int_{S^{n-1}} |f(s\eta)|~e^{su\eta .\omega}~d\eta $$ and
consequently, \Bea \Big(\int_{S^{n-1}}|T_{sz}f(\omega)|^2 d\omega
\Big)^{\frac {1}{2}} & \leq & \Big(\int_{S^{n-1}}
|f(s\eta)|^2~d\eta \Big)^{\frac {1}{2}}
\int_{S^{n-1}}e^{su\eta .\omega}~d \eta \\
& \leq & C~e^{- \frac{a}{2} s^2}~\frac{J_{\frac {n}{2}-1}(isu)}
{(isu)^{\frac{n}{2}-1}}\\
\Eea Notice that $ \Big(\int_{S^{n-1}}|Bf(z,\omega)|^2 d\omega
\Big)^{\frac {1}{2}}$ is bounded by \Bea && e^{-\frac
{1}{4}(u^2-v^2)} \Big( \int_{S^{n-1}}\big(\int_{0}^{\infty}
|T_{sz}f(\omega)|~e^{-\frac{1}{2} s^2}
s^{n-1}ds\big)^2~d\omega \Big)^{\frac{1}{2}} \\
&& \leq e^{- \frac {1}{4}(u^2-v^2)}\int_{0}^{\infty} \Big(
\int_{S^{n-1}}|T_{sz}f(\omega)|^2~d\omega \Big)
^{\frac{1}{2}}~e^{- \frac{1}{2} s^2}s^{n-1}ds \Eea where the last
inequality is achieved using Minkwoski's integral inequality.

Now, using the above estimates, we get \Bea
\Big(\int_{S^{n-1}}|Bf(z,\omega)|^2 d\omega \Big)^{\frac {1}{2}} &
\leq &  C~e^{- \frac {1}{4}(u^2-v^2)}\int_{0}^{\infty} e^{-
\frac{1}{2}(1+a)s^2} \frac{J_{\frac {n}{2}-1}(isu)} {(isu)^{\frac
{n}{2}-1}}s^{n-1}ds \\
& \leq & \frac {C}{(1+a)^{\frac {n}{2}}}~
\exp\Big(-\big(\frac{u^2-v^2}{4}\big) \Big)~\exp\Big(\frac
{u^2}{2(1+a)}\Big) \\
& = & \frac {C}{(1+a)^{\frac {n}{2}}}~\exp \Big(\frac {v^2+\mu
u^2}{4} \Big). \Eea Replacing $f$ by $\widehat{f}$ and using the
fact that $ B\widehat{f}(z,\omega)=
(2\pi)^\frac{n}{2}Bf(-iz,\omega)$, we also  get \Bea
\Big(\int_{S^{n-1}}|Bf(z,\omega)|^2~d\omega \Big)^{\frac {1}{2}}
\leq \frac {C}{(1+a)^{\frac {n}{2}}}~\exp\big(\frac {u^2+\mu
v^2}{4}\big) \Eea Hence the proposition is proved.
\end{proof}

\begin{thm} For a function $f $ on ${\mathbb{R}}^n$ satisfying the
conditions (5.1) and (5.2) let $ Bf(z,\omega) =
\sum_{k=0}^{\infty}d_k(\omega)z^k $ be the Taylor series expansion
of the Bargmann transform. Then \Bea \int_{S^{n-1}}
|d_k(\omega)|^2~d\omega  \leq C 2^{-k}
\frac{k^{-\frac{1}{2}}}{\Gamma(k+1)}~ {\mu}^{\frac {k}{2}}. \Eea
\end{thm}
\begin{proof} The proof of this theorem can be reduced to the scalar valued
case treated in \cite{V}. Indeed, for any normalised
$ g \in L^2(S^{n-1}) $ the scalar valued function
$$ F_g(z) = \int_{S^{n-1}} Bf(z,\omega)g(\omega) d\omega $$  is an entire
function satisfying estimates stated in Proposition 5.1. The
arguments in \cite{V}                                                                      lead to estimates for the integral
$$ \int_{S^{n-1}} d_k(\omega) g(\omega) d\omega.$$ Taking supremum over all
such $ g $ we get the required estimates.
\end{proof}

In order to apply the above estimates to prove Theorem we need the
following result which shows that the $ L^2(S^{n-1}) $ norms of $
d_k(\omega) $ can be expressed in terms of Laguerre coefficients
of the spherical harmonic components of $ f $ restricted to the
unit sphere.

\begin{thm} For $f \in L^2(\mathbb{R}^n),~if~ Bf(z,\omega)=
\sum_{k=0}^{\infty}d_{_k}(\omega)z^k$ is the Taylor series
expansion, then for all $ k \geq 1 $,
$$ \int_{S^{n-1}}|d_{2k}(\omega)|^2 d\omega ~=~ 2^{-n-4k+2}
\sum_{m=0}^k \sum_{j=1}^{d_{2m}} \frac {|(\widetilde{f}_{2m,j}~,
~\psi_{k-m}^{\frac{n}{2}+2m-1} )|^2}{\big(\Gamma(\frac{n}{2}+k+m)
\big)^2},$$
$$ \int_{S^{n-1}}|d_{2k+1}(\omega)|^2 d\omega ~=~
2^{-n-4k} \sum_{m=0}^k \sum_{j=1}^{d_{2m+1}} \frac
{|(\widetilde{f}_{2m+1,j}~,~\psi_{k-m}^{\frac{n}{2}+2m})|^2}
{\big(\Gamma(\frac{n}{2}+k+m+1)\big)^2}.$$
\end{thm}

\begin{proof} We know that for every $z_1,z_2 \in \mathbb{C}$
$$
e^{iz_{_1}z_{_2}\eta.\omega} = \sum_{m=0}^\infty \sum_{j=1}^{d_m}
\frac {J_{\frac{n}{2}+m-1}(z_1z_2)}
{(z_1z_2)^{\frac{n}{2}-1}}Y_{mj}(\eta)~ Y_{mj}(\omega)$$ which in
particular implies that
$$ e^{sz\eta.\omega} = \sum_{m=0}^\infty \sum_{j=1}^{d_m}
\frac{J_{\frac{n}{2}+m-1}(-isz)}{(-isz)^{\frac{n}{2}-1}}
Y_{mj}(\eta)~Y_{mj}(\omega)$$ and  thus $$
\int_{S^{n-1}}f(s\eta)~e^{sz\eta.\omega}~d\eta = \sum_{m=0}^\infty
\sum_{j=1}^{d_m} f_{mj}(s) \frac {J_{\frac{n}{2}+m-1}(-isz)}
{(-isz)^{\frac{n}{2}-1}}~Y_{mj}(\omega) .$$ Now, \Bea Bf(z,\omega)
&=& e^{-\frac{1}{4}z^{^2}} \int_{\mathbb{R}^n}f(x)~
e^{-\frac{1}{2}|x|^2}e^{zx.\omega}~dx\\
&=& e^{-\frac{1}{4}z^2} \int_0^\infty \int_{S^{n-1}} f(s\eta)
e^{-\frac{1}{2}s^2}e^{sz\eta.\omega}s^{n-1}ds~d\eta \\
&=& e^{-\frac{1}{4}z^2} \sum_{m=0}^\infty \sum_{j=1}^{d_m}
\Big(\int_0^\infty f_{mj}(s)~ \frac {J_{\frac{n}{2}+m-1}(-isz)}
{(-isz)^{\frac{n}{2}-1}}~e^{-\frac{1}{2}s^2}s^{n-1}ds \Big)
Y_{mj}(\omega)\\
&=& \sum_{m=0}^\infty \sum_{j=1}^{d_m} (-iz)^m~(U_m
\widetilde{f}_{mj})(-z)~Y_{mj}(\omega)\\
&=& \sum_{m=0}^\infty \sum_{j=1}^{d_m} (-iz)^m~(U_m
\widetilde{f}_{mj})(z)~Y_{mj}(\omega) \Eea where for simplicity we
have written $ U_m $ in place of  $ U _{\frac{n}{2}+m-1}.$

If we write the power series expansion of $U_m\widetilde{f}_{mj}$
as \Bea (U_{_m}\widetilde{f}_{mj})(z) ~=~
\sum_{k=0}^{\infty}b_{k,mj}~z^{2k} \Eea then, as we saw in Section
4, \Bea b_{k,mj}=2^{-(\frac{n}{2}+m-1)}\frac{(-1)^k2^{-2k}}
{\Gamma(\frac{n}{2}+k+m)}(\widetilde{f}_{mj}~,
~\psi_k^{\frac{n}{2}+m-1}) \Eea Now,  $d_k(\omega) = \frac{1}{2\pi
i} \int_\gamma \frac{Bf(z,\omega)}{z^{k+1}}dz~$ implies that \Bea
d_{2k}(\omega)= (-1)^{\frac{n}{2}-1} \sum_{m=0}^k
\sum_{j=1}^{d_{2m}} (-i)^m
\big(b_{k-m;~2m,j}\big)~Y_{2m,j}(\omega) \Eea which implies, \Bea
\int_{S^{n-1}}|d_{2k}(\omega)|^2 d\omega = \sum_{m=0}^k
\sum_{j=1}^{d_{2m}}|b_{k-m;~2m,j}|^2 \Eea and  therefore \Bea
\int_{S^{n-1}}|d_{2k}(\omega)|^2 d\omega = 2^{-n-4k+2}
\sum_{m=0}^k \sum_{j=1}^{d_{2m}} \frac
{|(\widetilde{f}_{2m,j}~,~\psi_{k-m}^{\frac{n}{2}+2m-1}
)|^2}{\big(\Gamma(\frac{n}{2}+k+m) \big)^2} \Eea A similar
calculation shows that \Bea \int_{S^{n-1}}|d_{2k+1}(\omega)|^2
d\omega = 2^{-n-4k} \sum_{m=0}^k \sum_{j=1}^{d_{2m+1}} \frac
{|(\widetilde{f}_{2m+1,j}~,~\psi_{k-m}^{\frac{n}{2}+2m}
)|^2}{\big(\Gamma(\frac{n}{2}+k+m+1) \big)^2} \Eea which completes
the proof of the theorem.
\end{proof}

Combining Theorems 5.2 and 5.3 we can prove Theorem 1.4. To see this, we
first observe that the Hecke-Bochner formula for the Hermite projections lead
to
\begin{prop}
$$ \|P_{2k}f\|_2^2 = 2 \sum_{m=0}^{k} \sum_{j=1}^{d_{2m}}
\Big(\Gamma(k-m+1)~ \Gamma(\frac{n}{2}+k+m)\Big)
\frac{|(\widetilde{f}_{2m,j}~,~\psi_{k-m}^{\frac{n}{2}+2m-1} )|^2}
{\big(\Gamma (\frac{n}{2}+k+m) \big)^2}.$$ A similar expression
holds for $ P_{2k+1}f $ also.
\end{prop}

We note the similarity between the expression for $
\|P_{2k}f\|_2^2 $ and the $ L^2(S^{n-1}) $ norms of $
d_{2k}(\omega).$ We therefore, rewrite expression for $
\|P_{2k}f\|_2^2 $ as
$$ 2 \sum_{m=0}^{k} 2^{2m} \sum_{j=1}^{d_{2m}} c(k,m)
\Big(2^{-2k}\Gamma(2k+1)~\frac{|(\widetilde{f}_{2m,j},
\psi_{k-m}^{\frac{n}{2}+2m-1})|^2}{\big(\Gamma
\big(\frac{n}{2}+k+m) \big)^2} \Big).$$ where
$$ c(k,m) = \Big(2^{2(k-m)}\frac{\Gamma(k-m+1)~\Gamma(\frac{n}{2}+k+m)}
{\Gamma(2k+1)}\Big).$$ Stirling's formula for gamma functions show
that $ c(k,k) = O(k^{n/2-1}).$ In general, we have

\begin{lem} For any $ 0 \leq m \leq k $ we have $ c(k,m) = O(k^{(n-1)/2}).$
\end{lem}
\begin{proof} We show that for all $ k \geq 1$ and
 $0\leq m \leq k$,
$$ 2^{2(k-m)}\frac{\Gamma(k-m+1)~\Gamma(k+m+\frac{n}{2})}{\Gamma(2k+1)}
~=~ O(k^{\frac{n}{2}-\frac{1}{2}}).$$  For this, fix $k \in \N $
and consider $ c(k,m).$ When  $m=k$, $ c(k,k)=
\frac{\Gamma(2k+\frac{n}{2})}{\Gamma(2k+1)} $ and by Stirling's
formula, for large $ k $ it behaves like
$$ \frac{(2k+\frac{n}{2}-1)^{2k+\frac{n}{2}-\frac{1}{2}}~
e^{-(2k+\frac{n}{2}-1)}} {(2k)^{2k+\frac{1}{2}}e^{-2k}} =
 \frac{(1+\frac{\frac{n}{2}-1}{2k})^{2k}(2k+\frac{n}{2}-1)^
{\frac{n}{2}-\frac{1}{2}}~e^{-(\frac{n}{2}-1)}}{{(2k)^{\frac{1}{2}}}}.$$
As $(1+\frac{\frac{n}{2}-1}{2k})^{2k}~\leq~e^{\frac{n}{2}-1},
c(k,k)~=~O\big(k^{\frac{n}{2}-1}\big).$

Now for $ 0~\leq~m~<~k$, consider \Bea c(k,m)&=&
2^{2(k-m)}\frac{\Gamma(k-m+1)~\Gamma(k+m+\frac{n}{2})}
{\Gamma(2k+1)}\\
&\sim& 2^{2(k-m)}~ \frac{(k-m)^{k-m+\frac{1}{2}}~e^{-(k-m)}
(k+m+\frac{n}{2}-1)^{k+m+\frac{n}{2}-\frac{1}{2}}~
e^{-(k+m+\frac{n}{2}-1)}} {(2k)^{2k+\frac{1}{2}}e^{-2k}}\\
&=& \frac{e^{-(\frac{n}{2}-1)}}{\sqrt{2}}~ \Big( \frac{2^{-2m}
(k-m)^{k-m+\frac{1}{2}}(k+\frac{n}{2}+m-1)^
{k+\frac{n}{2}+m-\frac{1}{2}}} {k^{2k+\frac{1}{2}}} \Big) \\
&\leq& C_n~e^{-(\frac{n}{2}-1)}2^{-2m}\Big(1-\frac{m}{k}\Big)^k~
\Big(1+\frac{m+\frac{n}{2}-1}{k}\Big)^k \Big \{
\frac{1+\frac{m+\frac{n}{2}-1}{k}}{1-\frac{m}{k}}\Big \}^m ~
k^{\frac{n}{2}-\frac{1}{2}}\\
&\leq& C_n~e^{-(\frac{n}{2}-1)}2^{-2m}~e^{-m}
e{^{m+\frac{n}{2}-1}} \Big\{
\frac{1+\frac{m+\frac{n}{2}-1}{k}}{1-\frac{m}{k}}
\Big \}^m ~k^{\frac{n}{2}-\frac{1}{2}}\\
&=& C_n~2^{-2m}\Big\{
\frac{1+\frac{m+\frac{n}{2}-1}{k}}{1-\frac{m}{k}}\Big \}^m
~k^{\frac{n}{2}-\frac{1}{2}} \Eea But, $ 2^{-2m}\Big\{
\frac{1+\frac{m+\frac{n}{2}-1}{k}}{1-\frac{m}{k}}\Big \}^m~\leq~1
$ if and only if $ 1+\frac{m+\frac{n}{2}-1}{k}~\leq~
4(1-\frac{m}{k}) $ which happens precisely when $
 m~\leq~\frac{1}{5}\big(3k-\frac{n}{2}+1\big) .$ Since for sufficiently large
$k~,~\frac{1}{5}\big(3k-\frac{n}{2}+1\big) ~\geq~
[\frac{k+1}{2}],$ it follows that for
$0~\leq~m~\leq~[\frac{k+1}{2}], c(k,m)
\leq~C_n~k^{\frac{n}{2}-\frac{1}{2}}.$ Now consider
$[\frac{k+1}{2}]~<m~<k .$ In this case \Bea c(k,m) &\leq&
C_n~e^{-(\frac{n}{2}-1)}2^{-2m}
\Big(1-\frac{m}{k}\Big)^{k-m}~\Big(\frac{k+m+\frac{n}{2}-1}{k}\Big)^{k+m}~
k^{\frac{n}{2}-\frac{1}{2}}\\
&\leq& C_n~e^{-(\frac{n}{2}-1)} 2^{-2m}
\Big(1-\frac{m}{k}\Big)^{k-m}~
\Big(\frac{k+m}{k}\Big)^{k+m}~\Big(1+ \frac{\frac{n}{2}-1}{k+m}
\Big)^{k+m}~k^{\frac{n}{2}-\frac{1}{2}}\\
&=& C_n~e^{-(\frac{n}{2}-1)} 2^{-2m} \Big(1-\frac{m}{k}
\Big)^{k-m}~\Big(1+\frac{m}{k}\Big)^{k+m}~\Big(1+
\frac{\frac{n}{2}-1}{k+m}
\Big)^{k+m}~k^{\frac{n}{2}-\frac{1}{2}}\\
&\leq& C_n~e^{-(\frac{n}{2}-1)} 2^{-2m}
\Big(\frac{1}{2}\Big)^{k-m}~
2^{k+m}~e^{\frac{n}{2}-1}~k^{\frac{n}{2}-\frac{1}{2}}\\
&=& C_n~k^{\frac{n}{2}-\frac{1}{2}}\Eea In the second last step,
we use the fact that $ m~>~[\frac{k+1}{2}] $ implies that $
\big(1-\frac{m}{k}\big) < -\frac{1}{2}.$
\end{proof}

The extra factor of $ 2^{2m} $ in the expression for $ \|P_{2k}f\|_2^2 $
suggests that we consider the operator $ T $ defined by
$$ Tf(r\omega) = \sum_{m=0}^\infty 2^{-\frac{m}{2}}(\sum_{j=1}^{d_m}
f_{mj}(r)~Y_{mj}(\omega)).$$ It is then clear that when $ f $ satisfies the
Hardy conditions  $ Tf $ satisfies
$$ \int_{L^2(S^{n-1})} |Tf(r\omega)|^2 d\omega \leq C e^{-\frac{1}{2}ar^2}.$$
A similar estimate is true for $ \widehat{Tf} $ as well.

\begin{thm} Suppose $ f $ satisfies the Hardy conditions with $ a = \tanh(2t).$
Then for any $ k = 0,1,2,... $ we have
$$ \|P_k(Tf)\|_2 \leq C (2k+n)^{\frac{n-2}{4}} e^{-(2k+n)t/2}.$$
\end{thm}

The theorem follows by using the estimates obtained in Theorem 5.2
along with the above lemma. Theorem 1.4 follows as a corollary to
Theorem 5.6 since for such functions $ \|P_kf\|_2 \leq C
\|P_k(Tf)\|_2 $, where $ C $ depends on the number of spherical
harmonic coefficients present in the expansion of $ f .$

\end{document}